\documentclass[centertags,12pt]{amsart}
\usepackage{latexsym}
\usepackage{amsthm}
\usepackage{amssymb}
\usepackage{color}
\usepackage[utf8]{inputenc}
\usepackage{mathrsfs}
\usepackage[english]{babel}
\usepackage{setspace}
\usepackage{ textcomp }
\usepackage{graphicx}
\usepackage{float}
\usepackage{tikz}
\usepackage{soul}
\usepackage{amsmath}
\graphicspath{ {images/} }
\usepackage[colorlinks=true,citecolor=black,linkcolor=black,urlcolor=blue]{hyperref}

\numberwithin{equation}{section}
\makeatletter
\renewcommand{\subsection}{\@startsection
	{subsection}{2}{0mm}{\baselineskip}{-0.25cm}
	{\normalfont\normalsize\bf}}
\makeatother

\newtheorem{theorem}{Theorem}[section]
\newtheorem{proposition}[theorem]{Proposition}

\newtheorem{lemma}[theorem]{Lemma}

\newtheorem{result}[theorem]{Result}
{\theoremstyle{definition}
	\newtheorem{definition}[theorem]{Definition}

	}
\theoremstyle{remark}
\newtheorem{remark}[theorem]{Remark}

\newcommand{\bfq}{{\mathbb F_q}}
\newcommand{\fqs}{{\mathbb F_{q^2}}}
\newcommand{\PGU}{{\rm PGU}}

\newcommand{\PSL}{{\rm PSL}}
\newcommand{\PGL}{{\rm PGL}}

\newcommand{\PG}{{\rm PG}}
\newcommand{\cX}{{\mathcal X}}
\newcommand{\cG}{{\mathcal G}}

\newcommand{\cF}{{\mathcal F}}
\newcommand{\cH}{{\mathcal H}}
\newcommand{\cU}{{\mathcal U}}

\newcommand{\cC}{{\mathcal C}}
\newcommand{\aut}{{\rm Aut}}

\newcommand{\cM}{{\mathcal{M}}}

\definecolor{Amaranto}{rgb}{0.9, 0.17, 0.31}
\definecolor{Borgogna}{rgb}{0.5, 0.0, 0.13}

\begin{document}
	
	\author[V. Pallozzi Lavorante]{Vincenzo Pallozzi Lavorante}
	
	\address{Dipartimento di Matematica Pura e Applicata, Universit\'a degli Studi di Modena e Reggio Emilia}
	\email{vincenzo.pallozzilavorante@unimore.it}
	
	\author[V. Smaldore]{Valentino Smaldore}
	\address{	Dipartimento di Matematica, Informatica ed Economia, Universit\`a degli Studi della Basilicata}
	\email{valentino.smaldore@unibas.it}
	\title{New hemisystems of the Hermitian surface}

	\begin{abstract}
		Constructing hemisystems of the Hermitian surface is a well known,
		apparently difficult, problem in Finite geometry. So far, a few
		infinite families and some sporadic examples have been constructed.
		One of the different approaches relays on the Fuhrmann-Torres maximal
		curve and provides a hemisystem in $PG(3,p^2)$ for every prime $p$ of
		the form $p=1+16a^2$. Here we show that this approach also works in
		$PG(3,p^2)$ for every prime $p=1+4a^2$.
		The resulting hemisystem gives rise to two weight linear codes and
		strongly regular graphs whose properties are also investigated.
	\end{abstract}
	
	\maketitle
	
	\begin{small}
		
		
		{\bf 2020 MSC:} {\color{black} Primary: 05B25. Secondary: 05E30, 51E20}.
		
	\end{small}

	\section{Introduction}

The Hermitian surface $\mathcal{U}_3$ of $PG(3,q^2)$ is the set of
all self-dual points of a non-degenerate unitary polarity of
$PG(3,q^2)$.
A generator of $\mathcal{U}_3$ is a line of $PG(3,q^2)$ entirely
contained in $\mathcal{U}_3$. The total number of generators of
$\mathcal{U}_3$ is $(q^3+1)(q+1)$ and through  any point
$P\in\mathcal{U}_3$ there exist exactly $q+1$ generators and they are
the intersection of $\mathcal{U}_3$ with its tangent plane at $P$.
Therefore, for any divisor $m$ of $q+1$, one can ask whether a
symmetric point-generator configuration for a family of generators
exists such that each point of $\mathcal{U}_3$ is incident with
exactly $(q+1)/m$ generators from the family.

In \cite{Segre} Segre proved
that such a symmetric point-generator configuration does not exist for
$m\neq 2$, and he introduced the concept of a hemisystem for the case
$m=2$. Therefore, a hemisystem of $\cU_3$ consists of
$\frac{1}{2}(q^3+1)(q+1)$ generators of $\cU_3$, exactly
$\frac{1}{2}(q+1)$ for each point on $\cU_3$. Segre exhibited a
hemisystem for $q=3$.

Hemisystems are interesting configurations which are connected with
important combinatorial objects such as strongly regular graphs,
partial quadrangles and 4-class imprimitive cometric
$Q$-antipodal association schemes that are not metric;  see
\cite{0,B1,B2}. Nevertheless, finding hemisystems is a challenging
problem.
The first infinite family
was constructed almost 50 years after by Cossidente and Penttila
\cite{0} who also found a new sporadic example in $\cU_3(5)$.  Later
on,  Bamberg, Giudici and Royle \cite{11} and \cite[Section 4.1]{22}
constructed more sporadic examples for $q=11,17,19,23,27$.
Recently several new infinite families of hemisystems appeared in the
literature. Bamberg, Lee, Momihara and Xiang \cite{B1} constructed  a
new infinite family of hemisystems on $\cU_3(q)$ for every $q\equiv -1
\pmod 4$ that generalize the previously known sporadic examples. Their
construction is based on  cyclotomic classes of $\mathbb{F}_{q^6}^*$
and involves results on characters and Gauss sums. Cossidente and
Pavese \cite{cp} constructed, for every odd $q$, a hemisystem of
$\cU_3$ left invariant by a subgroup of $\PGU(4,q)$ of order
$(q+1)q^2$.

The approach introduced in \cite{1} relays on the Fuhrmann-Torres
curve over $q^2$  naturally embedded in $\mathcal{U}_3$. Here term of
a curve defined over $q^2$ is used for a (projective, geometrically
irreducible, non-singular) algebraic curve $\cX$ of $PG(3,q^2)$.
Their construction provided a hemisystem of $\cU(3,q)$
whenever $q=p$ is a prime of the form $p=1+4a^2$ for an even integer
$a$. In this paper we investigate the analog construction for
$p=1+4a^2$  with an odd integer $a$, and show that it produces a
hemisystem , as well, for every such $p$.
We mention that a prime number $p$ of the form $p=1+4a^2$ with an
integer $a$ is called a Landau number. If Landau's conjecture is true,
that is there exist infinitely many Landau numbers, then an infinite
family of hemisystems is obtained.
	Our main result is stated in the following theorem.
	\begin{theorem}\label{th:1}
		Let $ p $ be a prime number where $ p =1 +4a^2 $ with an integer $ a $. Then there exists a hemisystem in the Hermitian surface $ \mathcal{U}_3 $ of $ \PG(3, p^2) $ which is left invariant by a subgroup of $ \PGU(4, p) $ isomorphic to $ \PSL(2, p)\times C_\frac{q+1}{2} $.
	\end{theorem}

	\section{Background on Hermitian surfaces, Maximal curves and Hemisystems}
	A canonical form of $\mathcal{U}_3$ is
\[X_0^{q+1}+X_1^{q+1}+X_2^{q+1}+X_3^{q+1}=0,\]
and the group of projectivities preserving $\mathcal{U}_3$ is
isomorphic to the projective unitary group $PGL(4,q)$ and it acts on
the points $\mathcal{U}_3$ as 2-transitive permutation group.
A hemisystem of $\cU_3$ consists of $\frac{1}{2}(q^3+1)(q+1)$ generators of $\cU_3$, exactly $\frac{1}{2}(q+1)$ of them through each point on $\cU_3$.
Up to a change of the projective frame in $ \PG(3, \fqs) $, the equation of $ \cU_3 $ may also be written in the form
	\[\cU_3 \colon X_1^{q+1}+2X_2^{q+1}-X_3^qX_0-X_3X_0^q=0.\]

As it is customary in algebraic geometry, the geometric objects are
viewed over the algebraic closure $\mathbb{K}$ of $q^2$.
An algebraic curve defined over $ \fqs $ means a projective, geometrically irreducible, non-singular algebraic curve $ \cX $ of $ \PG(3, q^2) $ viewed as a curve of $ \PG(3, \mathbb{K}) $, where $\mathbb{K}$ is the algebraic closure of $\mathbb{F}_q$. The curve $\cX$ is $\fqs$-maximal when the number $N_{q^2}$ of its points attains the Hasse-Weil upper bound, namely $N_{q^2}=q^2+1+2q g (\cX)$, where $g(\cX)$ is the genus of $\cX$.
In $\PG(2,\mathbb{K})$ with homogeneous coordinates $(x,y,z)$, the
Fuhrmann-Torres is the plane curve $\cF^+$ of genus
$\frac{1}{4}(q-1)^2$ with equation
    \[\cF^+\colon y^q-yz^{q-1}=x^\frac{q+1}{2}z^\frac{q+1}{2}.\]
The morphism \[\varphi \colon \cF^+ \to \PG(3, \mathbb{K}), \quad
(x,y,z)\mapsto (z^2,xz,yz,y^2 )\] defines an embedding (called natural
embedding) of $\cF^+$ which is a $q+1$ degree curve $\cX^+$  whose
points (including those defined over $\mathbb{K}$) are contained in
$\mathcal{U}_3$. In particular, $\cF^+$ is an $\fqs$-maximal curve.
 The twin Fuhrmann-Torres curve is defined
by the equation
\[\cF^-\colon y^q-yz^{q-1}=-x^\frac{q+1}{2}z^\frac{q+1}{2}.\]
and the above claims remain valid with respect to the same morphism. For more details see \cite{FT}.

Some useful properties of the Furhrmann-Torres curve, also valid for
any $\fqs$-maximal curve $\cX$ naturally embedded in $\mathcal{U}_3$,
are collected in the following results  obtained in \cite[Section
2,3,4]{1}.\\

	Thus, $\cX^+$ is a $q+1$ degree curve lying in the Hermitian surface $\cU_3$. Furthermore $\cX^+(\fqs)$ is partitioned in $\Omega$ and $\cX^+(\fqs) \setminus \Omega=\Delta^+$, where $\Omega$ is the set cut out on $\cX^+$ by the plane $\pi\colon X_1=0$. Note that $|\Omega|=q+1$ and $|\Delta^+|=\frac{1}{2}(q^3-q)$.

	Equivalently $\Omega$ is the intersection in $\pi$ of the conic $ \cC $ with equation $X_0X_3-X_2^2=0$ and the Hermitian curve $\cH(2,q^2)$ with equation $X_0^qX_3+X_0X_3^q-2X_2^{q+1}=0$.\\
	Moreover, the above properties hold true for when $^+$ is replaced by $^-$	and $\cX^-$ is a non singular model for the plane curve:
	\[\cF^- \colon y^q-yz^{q-1}=-x^\frac{q+1}{2}z^\frac{q+1}{2}. \]
	The curves $\cX^+$ and $\cX^-$ are isomorphic over $\fqs$ and $\Omega$ is the set of common points of $\cX^+$ and $\cX^-$.
	
	The curves $\cF ^\pm $ are known as the Fuhrmann-Torres maximal curves and they have their own specific interest.
		
We use classical terminology regarding rational curves. In particular, a (real) chord of $\cX$ is a line in $ \PG(3,q^2) $ which meets $ \cX(\fqs) $ in at least two distinct point, whereas an imaginary chord of $\cX$ is a line in $\PG(3,q^2)$ joining a point $P \in \cX(\mathbb{F}_{q^4})\setminus\cX(\fqs)$ to its conjugate, that is, its Frobenius image.
	
	\section{Previous results}
	
	We report a number of the results from \cite[Section 2,3,4]{2}, which are useful for our construction of Hemisystems based on the Fuhrmann-Torres maximal curves.
	
	From now $p \equiv 1 \pmod 4 $.
	
	Let $ \cH$ denote the set of all imaginary chords of $\cX$. Furthermore, for a point $P \in \PG(3,\fqs)$ lying in $\cU_3 \setminus \cX(\fqs)$, let $n_P(\cX)$ denote the number of generators of $\cU_3$ through $P$ which contain an $\fqs$-rational point of $\cX$.
	
	\begin{definition}
		A set $\cM$ of generators of $\cU_3$ is an \textit{half-hemisystem} on $\cX$ if the following properties hold:
		\begin{itemize}
			\item[(A)] Each $\fqs$-rational points of $\cX$ is incident with exactly $\frac{1}{2}(q+1)$  generators in $\cM$.
			\item[(B)] For any point $P \in \cU_3 \setminus \cX(\fqs)$ lying in $\PG(3, \fqs)$, $ \cM $ has as many as $\frac{1}{2}n_P(\cX)$ generators through $P$ which contain an $\fqs$-rational point of $\cX$.
		\end{itemize}
	\end{definition}
	
	Note that $\cM$ consists of $\frac{1}{2}(q+1)N_{q^2}$ generators and $\cH$ of $\frac{1}{2}(q^2+q)(q^2-q-2g(\cX))$ generators of $\cU_3$. Therefore $\cM \cup \cH$ has exactly $\frac{1}{2}(q^3+1)(q+1)$ generators of $\cU_3$.

	\begin{result}{\cite[Proposition 4.1]{1}}
		$\cM \cup \cH$ is a hemisystem of $\cU_3$.
	\end{result}

	Let $\mathfrak{G}$ be a subgroup of $\aut(\cX)$ and $o_1,\dots,o_r$ the $ \mathfrak{G} $-orbits on $\cX(\fqs)$.
	Moreover, for $1 \leq j \leq r$, let $\cG_j$ denote the set of all generators of $\cU_3$ meeting $o_j$. Note that $\mathfrak{G}$ leaves each $\cG_j$ invariant.
	
	\begin{result} \label{CDE}
		With the above notation, assume that the subgroup $\mathfrak{G}$ fulfils the hypothesis:
		\begin{itemize}
			\item[(C)] $\mathfrak{G}$ has a subgroup $\mathfrak{h}$ of index $2$ such that $\mathfrak{G}$ and $\mathfrak{h}$ have the same orbits $o_1,\dots,o_r$ on $\cX(\fqs)$.
			
			\item[(D)]For any $1\leq j \leq r$, $ \mathfrak{G} $ acts transitively on $\cG_j$ while $\mathfrak{h}$ has two orbits on $\cG_j$.
			
		\end{itemize}
		Let $P\notin \cX(\fqs)$ be a point lying on a generator in $\cG$, if
		\begin{itemize}
			\item[(E)] there is an element in $\mathfrak{G}_P$ not in $\mathfrak{h}_P$,
		\end{itemize}
		then $P$ satisfies \emph{(B)}.
	\end{result}
	
	Let $\cG$ be the set of all generators meeting $\cX^+$. From \cite[Lemma 5.1]{1} $\cG$ is  also the set of all generators meeting $\cX^-$. In particular, $\cG$ splits into two subset \begin{equation}\label{rk:G}
			\cG=\cG_1 \cup \cG_2,
		\end{equation}
		where $\cG_2$ is the set of the $(q+1)^2 $ generators meeting $\Omega$, while $\cG_1$ is the set of the $\frac{1}{2}(q^3-q)(q+1)$ generators meeting both $\Delta^+$ and $\Delta^-$.
		Thus, the following characterization of $\cG$ is very useful.
		
		\begin{result}{\cite[Lemma 5.3]{1}}\label{lm:G1}
			The generator set $\cG_1$ consists of all the lines $g_{u,v,s,t}$ spanned by the points $P_{u,v}=(1,u,v,v^2) \in \Delta^+ $ and $Q_{s,t}=(1,s,t,t^2) \in \Delta^-$ lying on  \[\cF \colon F(v,t)=(v+t)^{q+1}-2(vt+(vt)^q)=0\]	and \[u^{\frac{q+1}{2}}=v^q-v, \quad -s^\frac{q+1}{2}=t^q-t, \quad u^qs = (t-v^q)^2.\]
		\end{result}
	
\begin{result}{\cite[Lemma 5.4]{1}}
	$ \aut(\cF) $ contains a subgroup $\Psi \cong \PGL(2,q)$ that acts faithfully on the set $ \cF(\fqs)\setminus \cF (\bfq) $ as a sharply transitive permutation group.
\end{result}
\subsection{Automorphisms preserving $\cG$ and $\cX^+$}
	
In this subsection we recall the main results about the group-theoretic properties involving, $\cX^+$, $\cX^-$ and $\cG$; see \cite[Section 5]{1}.
The authors showed that $\Psi$ contains a subgroup $\Gamma$ which acts sharply transitively on $\cG_1$. Furthermore, $\Gamma$ has a unique index $2$ subgroup $\Phi$ such that \[\Phi \cong PSL(2,q) \times C_\frac{q+1}{2}.\]
In particular, $\Phi$ has two orbits on $\cG_1$, namely $\cM_1$ and $\cM_2$.

	In terms of subgroups of $\PGU(4,q)$ we have the following characterization.	
		\begin{result}\cite[Lemma 5.7]{1}
			The group $\PGU(4,q)$ has a subgroup $\mathfrak{G}$ with the following properties:
			\begin{itemize}
				\item[(i)]  $\mathfrak{G}$ is an automorphism group of $\cX^+$ and $\cX^-$;
				\item[(ii)]  $\mathfrak{G}$ preserves the point-sets $\Delta^+$, $\Delta^-$, $\Omega$ and $\cG_1$;
				\item[(iii)]  $\mathfrak{G}$ acts faithfully on $\Delta^+$, $\Delta^-$ and $\cG_1$;
				
				\item[(iv)]  $\mathfrak{G}$ acts on $\Omega$ as $\PGL(2,q)$ in its $3$-transitive permutation representation;
				\item[(v)] The collineation group induced by $\mathfrak{G}$ on $\pi$ is $\mathfrak{G}\textfractionsolidus Z(\mathfrak{G}) \cong \PGL(2,q)$ with $ Z(\mathfrak{G}) \cong C_\frac{q+1}{2} $.
			\end{itemize}
			Furthermore, $ 	\mathfrak{G} $ has an index $2$ subgroup $\mathfrak{h}$ isomorphic to $\PSL(2,q) \times C_\frac{q+1}{2}$.
		\end{result}
		With the above notation, in the isomorphism $\mathfrak{G} \cong \Gamma $, $\mathfrak{h}$ and $\Phi$ correspond.

		\begin{result}{\cite[Lemma 5.9]{1}.}
			The element of order $2$ in $ \mathfrak{h} $ are skew perspectivities, while those in $ \mathfrak{G} \setminus \mathfrak{h}   $ are homologies. Furthermore, the linear collineation $ \mathfrak{w} $, defined by
			\[\textbf{\emph{W}}:=\begin{pmatrix}
				1&0&0&0\\
				0&-1&0&0\\
				0&0&1&0\\
				0&0&0&1
			\end{pmatrix},\]
		interchanges $\cX^+$ with $\cX^-$ and the linear group generated by $\mathfrak{G}$ and $\mathfrak{w}$ is the direct product $ \mathfrak{G} \times \mathfrak{w} $.
		\end{result}
		
		\begin{result}
			$ \mathfrak{G} $ acts transitively on $\cG_2$ while $\mathfrak{h}$ has two orbits on $ \cG_2 $.
		\end{result}
From the result of this section, the following theorem follows				
\begin{theorem}\label{th:CD}
			Condition \emph{(C)} and \emph{(D)} are fulfilled for $\cX=\cX^+$, with $\Gamma=\mathfrak{G}$ and $\Phi=\mathfrak{h}$.
\end{theorem}

	More precisely,	$ \cG=\cG_1 \cup \cG_2 $ with $ \cG_1=\cM_1 \cup \cM_1'$ and $\cG_2=\cM_2 \cup \cM_2'$, where $\cG_1$ and $\cG_2$ are the $\mathfrak{G}$-orbits on $\cG$ whereas $\cM_1$, $\cM_1'$, $ \cM_2 $, $ \cM_2' $ are the $\mathfrak{h}$-orbits on $\cG_1$ and $\cG_2$ respectively. This notation fits with \cite[Section 5]{1}.

\subsection{Points satisfying Condition (E)}$ $
	
		The plane $\pi\colon X_1=0$ can be seen as the projective plane $\PG(2,q^2)$, with homogeneous coordinates $(X_0,X_2,X_3)$. Then $\cC$ is the conic of equation $X_0X_3-X_2^2=0$ and $\Omega$ is the set of point of $\cC$ lying in the (canonical Baer) subplane $\PG(2,q)$.\\
	The points in $\PG(2,q^2) \setminus \PG(2,q)$ are of three types with respect the lines of $\PG(2,q)$, i.e.
		\begin{itemize}
			\item[(I)] Point on a unique line disjoint from $\Omega$ which meets $\cC$ in two distinct points both in $\PG(2,q^2) \setminus \PG(2,q)$;
			\item[(II)] Point on a unique line meeting $\Omega$ in two distinct points;
			\item[(III)] Point on a unique line which is tangent to $\cC$ with tangency point on $\Omega$.
		\end{itemize}
		Points of type (I) - (II) and points in $\PG(2,q)$ satisfy condition (B), as can be readily seen in the next result.
		\begin{result}\label{th:I-II}
			If the projection of $P \in \cU_3$ on $\pi$ is a point $P'$ of type \emph{(I)} - \emph{(II)} or $P'\in \PG(2,q)$, then condition \emph{(E)} is fulfilled for $\cX=\cX^+$, $ \Gamma=\mathfrak{G} $ and $ \Phi=\mathfrak{h} $.
		\end{result}
		
		\section{Condition (B) for case (III) and $p \equiv 5 \pmod 8$}
	Condition (B) is rarely satisfied in Case (III), that is, for points $ P $ whose projection from $ X_\infty=(0,1,0,0) $ on $\pi$ is a point $ P' $ lying on a tangent $ l $ to $ \cC $.
	Our goal is to show that \cite[Theorem 7.1]{2}, proven for $q \equiv 1 \pmod 8$, remains true for \begin{center}
		$ p \equiv 5 \pmod 8$,
	\end{center}
extending their results to the case $p \equiv 1 \pmod 4$.	

From now on we assume $p \equiv 5 \pmod 8$.	
		\begin{theorem}\label{th:main}
			Condition $ (B) $ for Case $ (III) $ is satisfied if and only if the number $ N_q  $ of $ \bfq $-rational points of the elliptic curve with affine equation $ Y^2 = X^3 - X $ equals either $ q - 1 $, or $ q +3 $.
		\end{theorem}
		
		We need few steps before to prove Theorem \ref{th:main}. To begin with, we have to prove the following theorem.
		
		\begin{theorem}\label{th:III}
			Let $n_q$ be the number of $\xi \in \bfq $ for which $ f(\xi) = \xi^4-48\xi^2+64  $ is a square in $ \bfq $.
			Condition $ (B) $ for Case $ (III) $ is satisfied if and only if  $ n_q $ equals either $\frac{1}{2} (q+1)$ or $\frac{1}{2}(q-3)$.
		\end{theorem}
		
		The proof of Theorem \ref{th:III} is carried out by a series of lemmas.
		
		Since $q\equiv 5 \pmod 8$, unfortunately $2$ is not a square in $\bfq$. Therefore the proof is carried out significantly differently.
		
		Let $h$ and $-h$ be the roots of $2$ in $\fqs$.
		In particular we have that $h^q+h =0$ and $(\pm h)^{q+1} = -2$.\\
		Moreover $h^{1/2}=\alpha $, with $\alpha^2 =-1$. Thus, $\alpha \notin \square_q$ and $(1+\alpha)(1-\alpha)=2 \notin \square_q$.\\
		Since $ \mathfrak{G} $ is transitive on $\Omega$, the point $ O =(1, 0, 0, 0) $ may be assumed to be the tangency point of $ l $. Then $l$ has equation $ X_1 =0, X_3 =0 $, and $ P =(a_0, a_1, a_2, 0) $ with  $a_1\neq0$ and $a_1^{q+1}+2a_2^{q+1}=0$. If $ a_0 =0 $ then $ P =(0, d, 1, 0) $ with $ d^{q+1} +2 =0 $ and  his projection to $\pi$ is $ P'=(0, 0, 1, 0) $, which is a point in $ \PG(2, \bfq) $. By Theorem \ref{th:I-II} the case $ a_0 =0 $ can be dismissed and $ a_0 =1 $ may be assumed.\\
		Therefore, we may limit ourselves to the point $P=(a,b,0)$ such that $a^{q+1}+2 b^{q+1}=0$. In this case, the latter equation holds for $a=\pm h^2 $ and $b= h$. From this		
		\[P=(2\varepsilon, h, 0 ), \quad \mbox{ where } \varepsilon \in \{-1,1\}\]
		and we can carry out the computation simultaneously.
		
		\subsection{Case of $\mathcal{G}_1$}
		
		We keep up our notation $P_{u,v}=(u,v,v^2)$ for a point in $\Delta^+$. The following lemmas are the analogues of those in \cite[section 7.1]{2} for the case
		
		\begin{lemma}
			Let $v \in \fqs \setminus \bfq $. Then there exists $u \in \fqs$ such that the line joining $P$ at $P_{u,v}$ is a generator of $\mathcal{U}_3$ if and only if
			
			\begin{equation}\label{v}
				(v^2+2hv)^\frac{q+1}{2} = 2\varepsilon(v^q-v)
			\end{equation}
			
			If \emph{(\ref{v})} holds, then $u$ is uniquely determined by $v$.
		\end{lemma}
		\begin{proof}
			The line $l=PP_{u,v}$ is a generator if and only if $P_{u,v}$ lies on the tangent plane to $\mathcal{U}_3$ at $P$. This implies
			\begin{equation}\label{u}
				u=\frac{v^2+2hv}{2\varepsilon}.
			\end{equation}
			and since $P_{u,v} \in \Delta^+$ then $u^\frac{q+1}{2}=v^q-v$ and $l$ is a generator.
			The converse follows from the proof of \cite[Lemma 7.4]{1}.
		\end{proof}
		
		Lemma \ref{v} can be extended to $Q_{s,t} \in \Delta^-$ provided that $u,v$ are replaced by
		
		\begin{equation} \label{t}
			(t^2+2ht)^\frac{q+1}{2}= -2\varepsilon(t^q-t)
		\end{equation}
		e
		\begin{equation} \label{s}
			s=\frac{t^2+2ht}{2\varepsilon}.
		\end{equation}

		Furthermore $ P,P_{u,v} $ and $ Q_{s,t} $ are collinear if and only if
		\begin{equation}
			\begin{cases}
				\begin{aligned}
					&2\varepsilon(t^2-v^2) =t^2u-v^2s \\
					&vt-h(v+t)=0
				\end{aligned}
				
			\end{cases}
		\end{equation}
		
		It follows.
		
		\begin{lemma}\label{lm:coll}
			Let $ v, t \in \fqs \setminus \bfq  $ with $ F(v, t) =0 $. If the line through $ P_{u,v} \in \Delta^+ $ and $ Q_{s,t} \in \Delta^- $ is a generator through $P$, then \begin{equation}\label{eq:coll}
				vt-h(v+t)=0.
			\end{equation}
		\end{lemma}
		
		We now count the number of generator in $\mathcal{G}_1$ which pass through $P$.
		
		\begin{lemma} \label{lm:soluz}
			Equation \emph{(\ref{v})} has exactly $ \frac{1}{2}(q+1) $ solution in $ \fqs \setminus \bfq $.
		\end{lemma}
		\begin{proof}
			Let $ r=vh^{-1} $. We obtain:
			\[(r^2+2r)^\frac{q+1}{2} = \varepsilon h(r^q+r).\]
			Hence,
			\begin{equation*}
				(r^2+2r)^\frac{q^2-1}{2}=-1
			\end{equation*}
			and then $r^2+2r$ is a non-square of $\fqs$.
			Thus, there exists $z \in \fqs$ such that $r^2+2r=h z^2$.  Now the system is
			\begin{equation}\label{z2}
				\begin{cases}
					\begin{aligned}
						&h z^2 = r^2+2r \\
						&\alpha h z^{q+1} = \varepsilon h(r^q + r)
					\end{aligned}
				\end{cases}
			\end{equation}
			where $\alpha=h^{(q-1)/2}$.
			Let $\lambda = zr^{-1}$,
			\begin{equation}\label{sL}
				\begin{cases}
					\begin{aligned}
						&h \lambda^2r = (r+2) \\
						&\alpha(\lambda r)^{q+1} = \varepsilon(r^q + r)
					\end{aligned}
				\end{cases}
			\end{equation}
			Since $r = 2/(h\lambda^2-1)$ we obtain
			\begin{equation}\label{lambda}
				4 \alpha\lambda^{q+1}-2\varepsilon(h\lambda^{2}-1)-2\varepsilon(h\lambda^2-1)^q=0
			\end{equation}
			Now if $\lambda=\lambda_1+ h \lambda_2$, with $\lambda_1$, $\lambda_2 \in \bfq$, equation (\ref{lambda}) reads
			\begin{equation}\label{eq:l12}
				\alpha\lambda_1^2-2\alpha\lambda_2^2-4\varepsilon \lambda_1\lambda_2+4\varepsilon=0.
			\end{equation}
			Since the determinant of the matrix associated to the quadratic form \eqref{eq:l12} is $-8\varepsilon$, that quadratic form is the equation of an irreducible conic of $\PG(2,q)$. Thus, we have exactly $q+1$ solutions $\lambda$ of \eqref{lambda}.
			
			Every solution $v$ of (\ref{v}) is in $\fqs$. In fact, if $(hr)^q =hr$ then $r^q=-r$ and  $\lambda r =0$, which contradicts the first equation of (\ref{sL}).
		\end{proof}

		\begin{lemma}\label{quad}
			For every solution $v=v_1+h v_2$ of \eqref{v}, \[\varepsilon v_2+\frac{\alpha}{2}(v_1 v_2 + v_1) \notin \square_q,\] where $\alpha=h^\frac{q-1}{2}$ is a non-square of $\bfq$.	
		\end{lemma}
		\begin{proof}
			Consider System \eqref{z2} and let  $z=z_1+hz_2$ and $r=r_1+hr_2$. Then
			\begin{equation}
				\begin{cases}
					\begin{aligned}
						&z_1^2+2 z_2^2=2r_1r_2+2r_2\\
						&\alpha  z_1^2-2 \alpha z_2^2 = 2 \varepsilon r_1.
					\end{aligned}
				\end{cases}
			\end{equation}
			Summing up gives
			\[\alpha z_1^2 = \alpha r_2 (r_1 +1)+\varepsilon r_1\]
			
			Since $\alpha^2 =-1$ and $q \equiv 5 \pmod 8 $, it follows $ \alpha \notin \square_q$ and then
			\[\alpha r_2 (r_1 +1)+\varepsilon 	r_1\]
			
			is a non-square of $\bfq$. With $v_2=r_1$ and $v_1=2 r_2$ we obtain
			\[\varepsilon v_2+\frac{\alpha}{2}(v_1 v_2 + v_1) \notin \square_q.\]
		\end{proof}
		
		Our next step is to characterize the generators of $\mathcal{G}_1$ through $P$.
		
		To begin with, we need some notions of number theory, which would allow us to simplify the notation we will use.
		Note that $(2+h)^\frac{q+1}{2}=\lambda h$, where,
		\begin{equation}\label{eq:hou}
		    \lambda=(2+h)^\frac{q+1}{2}h^{-1}=[(1+h)h]^\frac{q+1}{2}h^{-1}=(1+h)^\frac{q+1}{2}h^\frac{q-1}{2}
		\end{equation}
		Since \[ \lambda^2=(1+h)^{q+1}2^\frac{q-1}{2}=(1+h)(1-h)(-1)=1\]
		we have $\lambda=\pm 1$.
		Applying the Frobenius map to \eqref{eq:hou} gives
		\[\lambda=(1-h)^\frac{q+1}{2}(-h)^\frac{q-1}{2}.\]
		Hence $\lambda$ is independent of the choice of $h$ as a square root of $2$.
		\begin{proposition}\label{pr:hou}
We have \[\lambda=\begin{cases}
    \begin{aligned}
        1& ,  \quad q \equiv 13 &\pmod {16} \\
         -1&, \quad q \equiv 5 &\pmod {16}
    \end{aligned}
\end{cases}\]
\end{proposition}
\begin{proof}
        See Appendix A.
\end{proof}		
		Let
		\begin{equation}
			\chi :=\begin{cases}
				-1, \mbox{ if either } \varepsilon = 1 \mbox{ and } q \equiv 13 \pmod{16} \mbox{ or } \varepsilon = -1 \mbox{ and } q \equiv 5 \pmod{16}\\
				1, \mbox{ if either } \varepsilon = 1 \mbox{ and } q \equiv 5 \pmod{16} \mbox{ or } \varepsilon = -1 \mbox{ and } q \equiv 13 \pmod{16}
			\end{cases}
		\end{equation}
		
		Furthermore,
		\[v_0:= -2(h-2\chi), \quad u_0:=\frac{4}{\varepsilon}(2-h\chi) \]
		
		and
		\[t_0:=-2(h+2\chi), \quad s_0:=\frac{4}{\varepsilon}(2+h\chi)  \]
		Since  $h= \varepsilon (2+\chi h)^\frac{q+1}{2}$,
		\begin{equation}
			v_0^q-v_0= 4 h = u_0^\frac{q+1}{2}
		\end{equation}
		and
		\[u_0^qs_0=16(2-h \chi)^2=(t_0-v_0^q)^2\]
		
		Furthermore,
		\[ (v_0+t_0)^{q+1}=-32=2(t_0v_0+(t_0v_0)^q)\]	
		
		Therefore,  $F(v_0,t_0)=0$. Thus, from Theorem \ref{lm:G1}, the line trough $P_{u_0,v_0}$ and $Q_{s_0,t_0}$ is a generator $g_0 \in \mathcal{G}_1$.\\
		Moreover the following hold:
		\[u_0=\frac{v_0^2+2 h v_0}{2 \varepsilon}, \quad s_0=\frac{t_0^2+2 h t_0}{2 \varepsilon} \]
		
		showing that $g_0$ pass through $P$.
		
		We show how each generator $g$ passing through P can be obtained from $ g_0 $. Since $g=P_{u,v}Q_{s,t}$ passes through $P$, then, by Lemma \ref{lm:coll}, $F(v,t)=0$ and $vt=h(v+t)$.
		Now for $\alpha,\beta,\gamma$ and $\delta \in \bfq$, with $\alpha \delta-\beta\gamma \ne 0$, write
		\begin{equation*}\label{key}
				v=\frac{\alpha v_0+\beta}{\gamma v_0 + \delta}, \quad t=\frac{\alpha t_0+\beta}{\gamma t_0+\delta}.
		\end{equation*}
		From $v_0 t_0=-8$ and $v_0+t_0=-4 h$, we may write Equation \eqref{eq:coll} as
		\begin{equation}\label{eq:doppie}
			\begin{aligned}
				8\alpha\gamma&=2\alpha\beta+\beta\delta \\
			\beta^2&=8(\alpha^2-\alpha \delta-\beta\gamma)
			\end{aligned}
		\end{equation}
		Our aim is to show that these equations hold if and only if $\alpha,\beta,\gamma$ and $\delta$ depend on a unique parameter $\xi \in \bfq \cup \{\infty\}$.
		To begin with, let $\delta\ne 0$. Then $\alpha \ne 0$. The first equation in \eqref{eq:doppie} forces \[\gamma=\frac{(2\alpha+1)\beta}{8\alpha}.\]
		Together with the other equation, we have
		\[8\alpha^3-3\alpha\beta^2-8\alpha^2-\beta^2=0.\]
		Let $\xi=\beta \alpha^{-1}$. This implies $\alpha^2(8\alpha -3\alpha \xi^2-8-\xi^2)=0$. Therefore
		\[\alpha=\frac{\xi^2+8}{8-\xi^2},\]
		and the assertion follows for $\delta \ne 0$.
		For $\delta=0$ we may assume $\beta=1$. If $\alpha\neq0$ then $\gamma=1/4$ and  $8\alpha^2=-1$, which is impossible as $-1$ is a square in $\bfq$ while $8$ is not.
		When $\delta=\alpha=0$ and $\beta=1$, then $\gamma=\frac{-1}{8}$.
		
		Therefore,
		\begin{equation}\label{eq:vxi}
			v = v_\xi =\frac{(\xi^2+8)v_0+(\xi^2+8)\xi}{\frac{\xi}{8}(-\xi^2+24)v_0+8-3\xi^2}, \quad v_\infty=\frac{1}{-\frac{1}{8}v_0}=-2(h+2\chi).
		\end{equation}
		
		and the determinant of the associated fractional linear map equals	
		\begin{equation}
			\det(\xi)=\frac{(\xi^2+8)(\xi^4-48 \xi^2+64)}{8}, \quad \det(\infty)=(8)^{-1}
		\end{equation}
		
		These equations remains true for  $t_0$ and $t$:
		\begin{equation}\label{eq:txi}
			t = t_\xi =\frac{(\xi^2+8)t_0+(\xi^2+8)\xi}{\frac{\xi}{8}(-\xi^2+24)t_0+8-3\xi^2}, \quad t_\infty=\frac{1}{-\frac{1}{8}t_0}=-2(h-2\cX).
		\end{equation}

		Next, we show that \eqref{quad} imposes a condition on $\xi$ in  \eqref{eq:vxi}.

		\begin{lemma}\label{prin}
			$\xi^2+8$ is a square in $\bfq$.
		\end{lemma}
		\begin{proof}
			To use Lemma \ref{quad} we rewrite $ \varepsilon v_2+\frac{\alpha}{2}(v_1 v_2 + v_1) $ in terms of $\xi$. This requires a certain amount of straightforward and tedious computations that we omit.
			From \eqref{eq:vxi}, we have
			\begin{equation}
				v=\frac{4(\xi^2 + 8)}{\chi16 -\chi 2 \xi^2 + h (8 -\chi 8 \xi +\xi^2)}
			\end{equation}
			and
			\begin{equation}
				v_1=\frac{-4 (\chi 16 -\chi 2 \xi^2) (8 + \xi^2)}{k}, \quad v_2 = \frac{-4 (8 + \xi^2) (8 -\chi 8 \xi + \xi^2)}{k}
			\end{equation}
			where $k=128 +\chi  256 \xi - 224 \xi^2 +\chi  32 \xi^3 + 2 \xi^4$.\\
			Then,
			\begin{equation}
				\varepsilon v_2+\frac{\alpha}{2}(v_1 v_2 + v_1)= \frac{2(1-\chi  \varepsilon \alpha) (8 + \xi^2) ((-16 + 16 \alpha) + (8 + 32 \alpha) \xi + (6 + 10 \alpha) \xi^2 +
					\xi^3)^2}{(64 - 128 \xi - 112 \xi^2 - 16 \xi^3 + \xi^4)^2}
			\end{equation}
		Note that $(1+\alpha)(1-\alpha)=2$ and that $1+\alpha \in \square_q$ if and only if $ q \equiv 13 \pmod{16}$. In fact, \[1+\alpha = \pm h^\frac{q+3}{4} \in \square_q \iff h^\frac{(q-1)(q+3)}{8}=1\]
		and in this case $1-\alpha $ is a non-square in $\bfq$.
		
		Since $\chi \epsilon=1$ when $ q \equiv 5 \pmod {16}$ and $\chi \epsilon=-1$ when $ q \equiv 13 \pmod {16}$, we get that $1-\chi\epsilon\alpha$ is always a square in $\bfq$.
		Hence $\xi^2+8 \in \square_q $.
		\end{proof}
		
		To state a corollary of Lemmas \ref{lm:soluz}, \ref{quad} and \ref{prin} , the partition of $\bfq \cup \{\infty\}$ into two subsets $\Sigma_1 \cup \{\infty\}$ and $\Sigma_2$ is useful where $x \in \Sigma_1 \cup \{\infty\}$ or $x \in \Sigma_2$ according as $x^2+8 \in \square_q$ or not.
		
		\begin{proposition}
			Let $P=(2\varepsilon,h,0) \in \mathcal{U}_3$ with $h^2 =2$. Then the generators in $\mathcal{G}_1$ through the point $P$ which meet $\cX^+$ are as many as $n_P=\frac{1}{2}(q+1)$. They are precisely the lines $g_\xi$ joining $P$ to $P_{u,v}=(u,v,v^2)$ with $u,v$ as in equation \eqref{u} and \eqref{eq:vxi}, where $\xi$ ranges over the set $\Sigma_1 \cup \{\infty\}$.
		\end{proposition}
		
		\subsection{Case of $\cG_2$}
		This case requires much less efforts. The tangent plane $\pi_P$ at $P=(2\epsilon,h,0)$ meets $\pi$ in the line $r$ of equation $2h^qY+Z=0$. Since $\cC$ has equation $Z=Y^2$ in $\pi$, the only common points of $r$ and $\cC$ are $(0:0:0)$ and $Q=(0:2h:8)$, with $Q \notin \Omega$ as $h \notin \bfq$.
		Then we have the following result.
		\begin{proposition}
			Let $P=(2\varepsilon,h,0) \in \cU_3$, with $h^2 = 2$. Then there is a unique generator through the point $P$ which meets $\Omega$, namely the line $l$ through $P$ and the origin $O=(0,0,0)$.
		\end{proposition}
	From now on, we denote with $\ell^+$ and $\ell^-$ the two generators through $P$ when $\varepsilon=1$ and $\varepsilon=-1$ respectively.
\subsection{Choice of $\cM_1$ and $\cM_2$}
	In this last subsection, we are going to choose $\cM_1$ and $\cM_2$ such that Condition (B) fulfilled.
	
	We have two different generators $g_0$, one for $\varepsilon=1$, the other for $\varepsilon=-1$:
	\[g_0^+ \mbox{ passing through } P^+(2,h,0)\]
		
		and
		\[g_0^- \mbox{ passing through } P^-(-2,h,0)\]
\begin{lemma}
	The generators $ g_0^+ $ and $g_0^-$ are in different orbits of $\Phi$.
\end{lemma}	
\begin{proof}
	The linear collineation associated to the matrix $\textbf{W}$ interchanges the two generators.
\end{proof}

Let $ r $ (resp. $ r' $) be the number of generators in $ \cM_1 $ (resp. $\cM_1'$) through the point $ P^+ $ that meet $\Delta^+$. Note that
\begin{equation}\label{eq:r+r'}
	r+r' = \frac{1}{2}(q+1).
\end{equation}
Similarly,

\begin{lemma}
	The generators $\ell^+$ and $\ell^-$ are in different orbits of $\Phi$.
\end{lemma}
\begin{proof}
	We use the same arguments of \cite[Lemma 7.14]{1}. Indeed, we exchange $(\sqrt{-2}b,b,0)$ and $(-\sqrt{-2}b,b,0)$ with $P^+$ and $P^-$ and the proof follows.
\end{proof}

We are ready to choose $\cM_1$ and $\cM_2$.
\begin{itemize}
	\item $ \cM_1 $ is the $\Phi$-orbit containing $g_0^+$.

	\item $\cM_2$ is the $\Phi$-orbit containing $\ell^+$ for $r < r'$ and $\ell^-$ for $r > r'$.
\end{itemize}

\begin{remark}\label{rk:f(x)}
As in \cite[Proposition 7.15]{1}, $r'$ is obtained counting the squares in the value set of the polynomial $f(\xi)$, defined in Theorem \ref{th:III}. More precisely, we obtain that the number of $\xi \in \bfq$ for which $f(\xi) \in \square_q$ equals $2r'-1$.
\end{remark}
Therefore we have the following proposition.
\begin{proposition}
	Condition \emph{(B)} for case \emph{(III)} holds if and only if \[r=\frac{1}{4}(q-1), \mbox{ and } r'=\frac{1}{4}(q+3)\]
	or
	\[r=\frac{1}{4}(q+3), \mbox{ and } r'=\frac{1}{4}(q-1)\]
\end{proposition}
\begin{proof}
	Note that $n_P=\frac{1}{2}(q+3)$ and that condition (B) holds if and only if half of them is in $\cM_1 \cup \cM_2$. The choices of $r$ and $r'$ are readily seen.
\end{proof}
Thus, Theorem \ref{th:III} follows.

Since the properties of the plane curve $\cC_4$

\[ Y^2=X^4-24\omega X^2+16\omega^2, \mbox{ with } \omega=2\]

depend only on $q \equiv 1 \pmod 4$, we also get  the proof of Theorem \ref{th:main}, that is Condition (B) in case (III) is satisfied if and only if the curve $\cC_3$
\[Y^2=X^3-X\]
has $q-1$ or $q+3$ points.
For the details, see \cite{1} at the end of  Section 7.
\section{conclusion}
We are in the position to work out the case $q=p$ when $p \equiv_4 1$. We write $p=\pi \bar{\pi}$, with $\pi \in \mathbb{Z}[i]$.\\
Here, $\pi$ can be chosen such that $\pi= \alpha_1 + i \alpha_2$ and $\alpha_1=1$. From \cite[Section 2.2.2]{3}, $N_p(\cC_3)=q+1-2\alpha_1$.
This implies that condition (B) in case (III) is satisfied if and only if
\[p=1+4a^2\quad \mbox{ and } \quad N_p(\cC_3)=q-1\]
Therefore, Theorem \ref{th:1} is a corollary of Theorem \ref{th:CD}, \ref{th:I-II} and \ref{th:III}.

\section{An application to strongly regular graphs}
 A \textit{strongly regular graph} with parameters $(v,k,\lambda,\mu)$ is a graph with $v$ vertices, each vertex lies on $k$ edges, any two adjacent vertices have $\lambda$ common neighbours and any two non-adjacent vertices have $\mu$ common neighbours. Strongly regular graphs have been obtained in different research areas in Combinatorics, in particular from Finite geometry.

 A strongly regular graph $\Gamma$ with parameters $((q^{3}+1)(q+1-m),(q^{2}+1)(q-m),q-1-m,q^{2}+1-m(q+1))$ may arise  from any $m$-regular system $\mathcal{S}$ on the Hermitian surface $\mathcal{H}(3,q^{2})$, $q$ odd, where vertices of $\Gamma$ are the lines lying on the surface but not contained in $\mathcal{S}$, and two vertices are adjacent if the lines are incident. Actually, this construction only works for $m=\frac{1}{2}(q+1)$ as it was pointed out in \cite{4}. In other words, each $m$-regular system on the Hermitian surface needs to be an hemisystem (according to Segre's result quoted in Introduction). Thus, every hemisystem gives rise to a strongly regular graph with the following parameters: $v=\frac{1}{2}(q^{3}+1)(q+1)$, $k=\frac{1}{2}(q^{2}+1)(q-1)$, $\lambda=\frac{1}{2}(q-3)$, $\mu=\frac{1}{2}(q-1)^{2}$. From this the spectrum of $\Gamma$ can be computed. The first eigenvalue is $k$, of multiplicity 1, and other two (the restricted eigenvalues) are:
 $$\theta_{1}= \textstyle\frac{1}{2}\big[(\lambda-\mu)+\sqrt{(\lambda-\mu)^{2}+4(k-\mu)}\big]=q-1,$$
 $$\theta_{2}=\textstyle\frac{1}{2}\big[(\lambda-\mu)-\sqrt{(\lambda-\mu)^{2}+4(k-\mu)}\big]=
 \textstyle\frac{1}{2}(-q^{2}+q-2),$$
 of multiplicity
 $$m_{1}=\textstyle\frac{1}{2}\Big[(v-1)-\frac{2k+(v-1)(\lambda-\mu)}{\sqrt{(\lambda-\mu)^{2}+4(k-\mu)}}\Big]=\textstyle\frac{1}{2}(q^{4}-q^{3}+2q^{2}-q+1),$$
 $$m_{2}=\textstyle\frac{1}{2}\Big[(v-1)+\frac{2k+(v-1)(\lambda-\mu)}{\sqrt{(\lambda-\mu)^{2}+4(k-\mu)}}\Big]=(q^{2}+1)(q-1)=2k,$$
 respectively.

 The hemisystem on the Hermitian surface $H(3,p^{2})$, $p=1+4a^{2}$ with an integer $a$, constructed in the present paper produces a strongly regular graph $\Gamma$ with the above parameters for $q=p$. We point out that, in the smallest case $p=5$, the graph $\Gamma$ has parameters $(378,52,1,8)$ and spectrum $52, 4^{273}, -11^{104}$. A comparison of $\Gamma$ with the Cossidente-Penttila strongly regular graph with the same parameters, shows that they are cospectral.  It is an open question whether these two strongly regular graphs are isomorphic.

	\section{Two-weight codes from Strongly regular graphs}\label{codes}

An $[n,k]$-linear code $C$ over the finite field $\bfq$ is a $k$-dimensional subspace of $\bfq^{n}$. Vectors in $C$ are called \textit{codewords}, and the weight $w(v)$ of $v\in C$ is the number of non-zero entries in $v$. A \textit{two-weight code} is an $[n,k]$-linear code $C$ such that $|\{w:\exists v\in C\setminus\{\underline{0}\}\hspace{2 mm} w(v)=w\}|=2$\\
For a subset $\Omega$ of $\mathbb{F}_{q}^k$, with $\Omega=-\Omega$ and $0\not\in \Omega$, define $G(\Omega)$ to be the graph whose vertices are the vectors of $\mathbb{F}_{q}^k$, and two vertices are adjacent if and only if their difference is in $\Omega$.  Moreover, let $\Sigma$ denote the set of points in $\mathrm{PG}(k-1,q)$ that correspond to the vectors in $\Omega$, i.e. $\Sigma=\{\langle \mathbf{v} \rangle\,:\, \mathbf{v} \in \Omega\}$.
An useful result connecting two-weight linear codes and strongly regular graphs is found in \cite{CK} which relies on projective $(n,k,h_{1},h_{2})$-sets, i.e. a proper, non-empty sets $\Sigma$ of $n$ points of the projective space $\mathrm{PG}(k-1,q)$ such that every hyperplane meets $\Sigma$ in either $h_{1}$ or $h_{2}$ points.
\begin{result}\label{CK1} \cite[Theorems 3.1 and 3.2]{CK}
	Let $\Omega$ and $\Sigma$ be defined as above. If $\Sigma=\{\langle \mathbf{v_i}\rangle\,:\, i=1,\ldots,n\}$ is a proper subset of $\mathrm{PG}(k-1,q)$ that spans $\mathrm{PG}(k-1,q)$, then the following are equivalent:
	\begin{itemize}
		\item[(i)] $G(\Omega)$ is a strongly regular graph;
		\item[(ii)] $\Sigma$ is a projective $(n,k,n-w_{1},n-w_{2})$-set for some $w_{1}$ and $w_{2}$;
        \item[(iii)]the linear code $ C=\{(\mathbf{x}\cdot \mathbf{v_1}, \mathbf{x}\cdot \mathbf{v_2}, \ldots, \mathbf{x}\cdot \mathbf{v_n})\,:\, \mathbf{x}\in \mathbb{F}_q^k\}$ (here $\mathbf{x}\cdot \mathbf{v}$ is the classical scalar product) is an $[n,k]$-linear two-weight code with weights $w_{1}$ and $w_{2}$.
	\end{itemize}
\end{result}
We point out that the hemisystem constructed in the present paper gives rise to a projective set. In fact, it is known that an $m$-regular system on the Hermitian surface also provides an $m$-ovoid $\mathcal{O}$ on the elliptic quadric $\mathcal{Q}^{-}(5,q)$, which is the image of $\mathcal{H}(3,q^{2})$ via the Klein correspondence.  Moreover, see \cite[Theorem 11]{5}, an $m$-ovoid on the elliptic quadric $\mathcal{Q}^{-}(5,q)$ is a projective $(m(q^{r+1}+1),6,m(q^{r}+1),m(q^{r}+1)-q^{r})$-set and it produces a strongly regular graph with parameters: $$(q^{6},m(q-1)(q^{3}+1),m(q-1)(3+m(q-1))-q^{2},m(q-1)(m(q-1)+1)).$$
 Since $m=\frac{1}{2}(q+1)$ we get an $srg(q^{6},\frac{1}{2}(q^{3}+1)(q^{2}-1),\frac{1}{4}(q^{4}-5),\frac{1}{4}(q^{4}-1))$, and the $\frac{1}{2}(q+1)$-ovoid $\mathcal{O}$ is a projective $(\frac{1}{2}(q^{3}+1)(q+1),6,\frac{1}{2}(q^{2}+1)(q+1),\frac{1}{2}(q^{3}-q^{2}+q+1))$-set, which gives the $[\frac{1}{2}(q^{3}+1)(q+1),6]$-linear two-weight code with weights $w_{1}=\frac{1}{2}q^{2}(q^{2}-1)$ and $w_{2}=\frac{1}{2}q^{2}(q^{2}+1)$.

The construction given in this paper, when $q=5$, gives rise to a $3$-ovoid on the elliptic quadric $\mathcal{Q}^{-}(5,5)$, and an $srg(15625,1512,155,156)$, i.e.  a projective $(378,6,78,53)$-set.
From Result \ref{CK1} it produces a $[378,6]$-linear two-weight code with $w_{1}=300$ and $w_{2}=325$.

\section*{Appendix A}
We provide a proof of Proposition \ref{pr:hou}. Since our proof relies on cyclotomic fields from algebraic number theory, we present it in the form of an appendix.

Let $\mathbb{Q}(\zeta_m)$ the cyclotomic field of $m$th roots of unity with $\zeta_m=e^{2\pi i /m} \in \mathbb{C}$.
In particular, the cyclotomic field $\mathbb{Q}(\zeta_{16})$ contains $\sqrt{2}$ as an integer. Let $\mathfrak{b}$ a prime ideal of $\mathbb{Q}(\zeta_{16})$ such that $\mathfrak{b}$ contains $p$ (i.e. $\mathfrak{b} \mid p$). The extension $\mathfrak{b} \mid p$ is unramified and $\mathbb{Z}[\zeta_{16}]/\mathfrak{b} \cong \mathbb{F}_{p^4}$; see \cite[Proposition 13.2.5]{IrRo} and \cite[Section 4.5]{XDH}. Note that $h=\pm \sqrt{2} \pmod {\mathfrak{b}}$. We may assume $h \equiv \sqrt{2} \pmod {\mathfrak{b}}$.

\begin{proof}[Proof of Proposition \ref{pr:hou}]
We do the computation for $q \equiv 13 \pmod{16}$, the proofs for the other cases being analogous.
\begin{equation*}
    \begin{aligned}
        (1+h)^\frac{q+1}{2}h^\frac{q-1}{2} &\equiv (1+\sqrt{2})^\frac{q+1}{2}(\sqrt{2})^\frac{q-1}{2} \pmod{\mathfrak{b}} \\
        & =(\sqrt{2}+2)^\frac{q+1}{2} \frac{1}{\sqrt{2}}\\
        &=(\zeta_8+\zeta_8^{-1}+2)^\frac{q+1}{2}\frac{1}{\sqrt{2}}\\
        &=(\zeta_{16}+\zeta_{16}^{-1})^{q+1}\frac{1}{\sqrt{2}}\\
        &\equiv (\zeta_{16}+\zeta_{16}^{-1})(\zeta_{16}^{13}+\zeta_{16}^{-13})\frac{1}{\sqrt{2}} \pmod{\mathfrak{b}}\\
       &\equiv (\zeta_{16}+\zeta_{16}^{-1})(\zeta_{16}^{-3}+\zeta_{16}^{3})\frac{1}{\sqrt{2}} \pmod{\mathfrak{b}}\\
       &=(\zeta_{16}^4+\zeta_{16}^{-2}+\zeta_{16}^2+\zeta_{16}^{-4})\frac{1}{\sqrt{2}}\\
       &=(\zeta_8+\zeta_8^{-1})\frac{1}{\sqrt{2}}=1
    \end{aligned}
\end{equation*}
\end{proof}

\section*{Acknowledgements}
The research of  Vincenzo Pallozzi Lavorante was partially supported  by the Italian National Group for Algebraic and Geometric Structures and their Applications (GNSAGA - INdAM).


\begin{thebibliography}{99}
  \bibitem{5} J. Bamberg, S. Kelly, M. Law, T. Penttila, Tight sets and $m$-ovoids of finite polar spaces, Journal of Combinatorial Theory A, 114(7), pp. 1293-1314, 2007.
  \bibitem{11} J. Bamberg, M. Giudici, G.F. Royle, Every flock generalized quadrangle has a hemisystem, Bulletin of the London Mathematical Society, 42 pp. 795–810, 2010.
  \bibitem{22} J. Bamberg, M. Giudici, G.F. Royle, Hemisystems of small flock generalized quadrangles, Designs Codes and Cryptography, 67, pp. 137–157, 2013.
  \bibitem{B1} J. Bamberg, M. Lee, K. Momihara, Q. Xiang, A new infinite family of hemisystems of the Hermitian surface, Combinatorica, 38, pp. 43–66, 2018.
  \bibitem{CK} R. Calderbank, W. M. Kantor, The geometry of two-weight codes,  Bulletin of the London Mathematical Society, 18, pp. 97-122, 1986.
  \bibitem{0} A. Cossidente, T. Penttila, Hemisystems on the Hermitian surface, Journal of the London Mathematical Society, 72(2), pp. 731-741, 2005.
  \bibitem{B2} A. Cossidente, Combinatorial structures in finite classical polar spaces, in: Surveys in Combinatorics 2017, in: LMS Lecture Note Series, vol.440, pp. 204–237, 2017.
  \bibitem{cp} A. Cossidente, F. Pavese, Intriguing sets of quadrics in $PG(5, q)$, Advances in Geometry, 17, pp. 339–345, 2017.
  \bibitem{FT} R. Fuhrmann, F. Torres, The genus of curves over finite fields with many rational points, Manuscripta Mathematica, 89, pp. 103–106, 1996.
  \bibitem{XDH} X. Hou, Lectures on Finite Fields, Graduate Studies in Mathematcs 190, American Mathematical Society, Providence, RI, 2018

  \bibitem{IrRo} K. Ireland and M. Rosen, A Classical Introduction to Modern Number Theory, 2nd ed., Springer, New York, 1990.

  \bibitem{2} G. Korchm\'{a}ros, F. Torres, Embedding of a Maximal Curve in a Hermitian Variety, Compositio Mathematica, 128 (1), pp. 95-113, 2001.
  \bibitem{1} G. Korchm\'{a}ros, G. P. Nagy, P. Speziali, Hemisystems of the Hermitian surface, Journal of Combinatorial Theory A, 165, pp. 408-439, 2019.
  \bibitem{Segre}B. Segre, Forme e geometrie hermitiane, con paricolare riguardo al caso finito, Annali di Matematica Pura ed Applicata, 1965, 70(1), pp. 1-201.
  \bibitem{3} J.P. Serre, Lectures on NX(p), CRC Press, Taylor e Francis, Boca Raton, 2011.
  \bibitem{4} J. A. Thas, Ovoids and spreads of finite classical polar spaces, Geometriae Dedicata, 10, pp. 135–143, 1985.
		\end{thebibliography}
	\end{document}